\documentclass[preprint,12pt]{elsarticle}
\usepackage{amsmath,amssymb,amsthm}
\usepackage{hyperref}
\usepackage{titlesec}
\usepackage{tikz}
\usepackage{xcolor}
\usetikzlibrary{arrows.meta, positioning, fit, backgrounds}
\titleformat{\paragraph}[runin]{\normalfont\bfseries}{}{0em}{}[]

\definecolor{factionBlue}{RGB}{37, 99, 235}
\definecolor{factionMagenta}{RGB}{217, 70, 239}
\definecolor{factionTeal}{RGB}{13, 148, 136}
\definecolor{factionGold}{RGB}{202, 138, 4}
\definecolor{factionViolet}{RGB}{124, 58, 237}

\newtheorem{theorem}{Theorem}
\newtheorem{corollary}{Corollary}
\newtheorem{lemma}{Lemma}

\theoremstyle{definition}
\newtheorem{definition}{Definition}

\newcommand{\gos}{\mathrm{gos}}
\newcommand{\Rsize}[1]{|R(#1)|}
\newcommand{\Rmat}{R}
\makeatletter
\def\ps@pprintTitle{\let\@oddhead\@empty \let\@evenhead\@empty
  \def\@oddfoot{\hfil\thepage\hfil}\let\@evenfoot\@oddfoot}
\makeatother

\begin{document}

\begin{frontmatter}

\title{Termination of Binary Trust-Gossip Dynamics:\\A Constructive No-Limit-Cycles Theorem}
\author{Nicholas Boichuk}
\address{Independent Researcher}
\ead{nicholas.boichuk3@nhs.net}

\begin{abstract}
In the binary trust-gossip dynamics, $n$ agents each hold a directed binary opinion, trust or distrust, of every other agent, and a gossip step lets one agent copy another agent's opinion of a third whenever the copier trusts the source. We prove that under any fair schedule of such steps, every trajectory reaches an absorbing state in finitely many steps, one in which no gossip step changes any opinion; in particular, there are no limit cycles. The proof is constructive and rests on a single descent measure: the number of ordered pairs of distinct agents linked by a chain of trust. Trust-adding gossip leaves this count unchanged and trust-removing gossip can only lower it, so it never rises, and it is zero only at the all-distrust state. From any non-absorbing state we exhibit a finite run of steps that reaches an absorbing state or strictly lowers the count; since the count is a non-negative integer at most $n(n-1)$, the process halts at an absorbing state.
\end{abstract}


\end{frontmatter}


\section{Introduction}

The companion paper \cite{absorbing} introduced the binary trust-gossip dynamics on $n$ agents, characterized its absorbing states completely, and counted them. To briefly describe the model: each agent holds a directed binary opinion (``trust'' or ``distrust'') of every other agent, encoded as a binary matrix $M$ in which the entry $M_{ij}$, written as $i(j) \in \{0, 1\}$, is agent $i$'s opinion of agent $j$. Each row therefore holds one agent's opinions of all the others, and the matrix is directed: $i(j)$ and $j(i)$ are independent. The diagonal entries of self-opinion $i(i)$ are left undefined throughout, following the convention of \cite{absorbing}. A gossip event $\gos(M, a(y), z)$ --- written $\gos(M, a, z, y)$ in \cite{absorbing} --- where $a, z, y$ are distinct agents, has the listener $z$ adopt the speaker $a$'s opinion of the target $y$ (where $a(y)$ is $a$'s opinion of $y$), if $z$ trusts $a$:
\begin{equation}
z(y) := a(y), \quad \text{provided } z(a) = 1.
\end{equation}
A state is \emph{absorbing} if no gossip event can change it, equivalently if for every triple $(a, z, y)$ of distinct agents, either $z(a) = 0$ or $z(y) = a(y)$ (no agent trusts another agent whom they disagree with).

The main results of the companion paper concerned the structure and count of these absorbing states. The Characterization Theorem of \cite{absorbing} establishes a bijection between absorbing states on $[n]$ and pairs consisting of a partition of $[n]$ with a non-empty subset of each block highlighted as ``core''. Enumerating these yields the OEIS sequences A143405 (labeled count) and A000219 (unlabeled count, the plane partitions); exhaustive enumeration confirms both counts for $n \leq 7$.

The companion paper also reports an empirical observation: for $n \leq 4$, exhaustive construction of the state-transition graph finds that every initial state reaches an absorbing state, with no limit cycles detected. This paper proves that observation holds for all $n$.

The descent argument here has a long lineage. That a game on a graph must halt because some integer quantity strictly drops is the principle behind
the chip-firing game of Bj\"orner, Lov\'asz, and Shor~\cite{bls}, its
directed version~\cite{bl-digraph}, and Dhar's abelian
sandpile~\cite{dhar}. Bond and Levine~\cite{bond-levine} showed one
commutativity axiom suffices: an abelian network halts, and the number
of moves is fixed, independent of their order. However, our binary trust-gossip is
not abelian: a gossip move overwrites the listener's opinion rather than
adding to it, and fires only when the listener trusts the speaker and
the two disagree, so different fair schedules reach different absorbing
states~\cite[Fig.~2]{absorbing}. The abelian halting theorem therefore does not
apply, and $|R(M)|$ is the monovariant we put in its place.

The model itself belongs to the domain of opinion dynamics, where the classical models
are continuous and average: DeGroot's pooling~\cite{degroot},
Hegselmann--Krause bounded confidence~\cite{hk}, and the gossip
algorithms of Boyd et al.~\cite{boyd}, in which gossip means pairwise
averaging toward a common mean. The voter
model~\cite{clifford-sudbury,holley-liggett} is binary but copies a
neighbor with no trust gating and settles at consensus. Coevolving models
such as DeGroot--Friedkin~\cite{jia} let the network change but keep
opinion and influence apart. Binary trust-gossip differs on each count:
a single binary matrix that is at once the opinions and the trust gating
their spread, updated by overwriting, terminating at the core--periphery
states of~\cite{absorbing} rather than at consensus.

\paragraph{Sketch of the proof.}
Everything rests on one count: $|R(M)|$, the number of ordered pairs of agents joined by a direct or indirect trust path. Gossip that adds a trust edge creates no new reachability, since the listener adopting trust of a target from a speaker could, by construction, already reach the target through the speaker, while gossip that removes an edge can only destroy reachability, so $|R(M)|$ never rises (Section~\ref{sec:basic}). The proof shows that from every non-absorbing state, some finite run of gossip either reaches an absorbing state or strictly lowers $|R(M)|$. Since $|R(M)|$ is a non-negative integer, the descent can repeat only a finite number of times. We fully enumerate the possible arrangements of the strongly-connected clusters of the reachability graph (Section~\ref{sec:dichotomy}). We show that a cluster which reaches a commonly-trusted outsider can be dissolved, strictly decreasing $|R(M)|$ (Section~\ref{sec:outward}). A cluster that reaches no outsider is instead completed into a locked faction (a fully-trusting core plus a periphery that copies it) after which there is no possible gossip event which could change their opinions (Section~\ref{sec:post-processing}). The agents left outside these factions can be absorbed into a faction they reach as a new peripheral, which lowers $|R(M)|$ (Section~\ref{sec:zipping}). Then, existence turns into inevitability: we only show that one finite path exists from every non-absorbing state to an absorbing one, and a fair schedule of firing moves cannot avoid it forever (Section~\ref{sec:fairness}).


\section{Setup and notation}\label{sec:setup}

Throughout this paper, the number of agents $n \geq 3$ is fixed to allow the gossip function to be defined.

\begin{definition}[Firing move; pre-state and post-state]\label{def:firing}
A \emph{firing move} is a gossip event $\gos(M, a(y), z)$ in which $a, z, y$ are distinct, $z(a) = 1$, and $z(y) \neq a(y)$. A firing move changes $M$: we call $M$ before the move the \emph{pre-state} and the result $M'$ the \emph{post-state}.
\end{definition}

As per \cite{absorbing}, a state is absorbing if and only if no firing move is available. We classify each firing move by the direction of the change to $M$:

\begin{definition}[$0 \to 1$ and $1 \to 0$ moves]\label{def:01moves}
A firing move $(a, z, y)$ is a \emph{$0 \to 1$ move} if it adds a direct trust edge from $z$ to $y$; that is, if $z(y) = 0$ and $a(y) = 1$, and therefore the move sets $z(y) := 1$. The move is a \emph{$1 \to 0$ move} if it removes a direct trust edge; that is, $z(y) = 1$ and $a(y) = 0$, therefore the move sets $z(y) := 0$.
\end{definition}

\begin{definition}[Reachability matrix]\label{def:reach} 
$\Rmat(M)$, the reachability matrix, is the transitive closure of $M$. It has $\Rmat_{uv} = 1$ exactly when $M$ contains a directed path $u \to \cdots \to v$ of length $\geq 1$; like $M$, it is defined only off the diagonal. We write $|R(M)|$ for its number of $1$-entries, so $|R(M)| \in \{0, 1, \ldots, n(n-1)\}$.
\end{definition}

\begin{definition}[$\Rmat$-cluster]\label{def:scc}
An \emph{$\Rmat$-cluster} of $M$, or \emph{cluster} for short, is a strongly-connected component of $\Rmat(M)$ with at least two agents: a maximal set $C$ with $|C| \geq 2$ in which every two distinct $u, v$ are mutually reachable ($\Rmat_{uv} = \Rmat_{vu} = 1$). Since $\Rmat_{uv} = 1$ means exactly that $u$ reaches $v$ in $M$, the mutually-reachable pairs are identical for $M$ and $\Rmat$, so the two split the agents into the same strongly-connected components. A lone agent, mutually reachable with no one, forms its own component but is not a cluster.
\end{definition}

\begin{definition}[Flat and progress]\label{def:flat}
A firing move is \emph{flat} if it preserves $|R(M)|$, and \emph{progress} if it strictly decreases $|R(M)|$. Corollary~\ref{cor:R-monotone} will show that these are the only two possibilities, i.e. that $|R(M)|$ is monotone non-increasing under every firing move.
\end{definition}

Figure~\ref{fig:example} shows a small $n = 7$ example state that the later step-by-step figures below refer to.

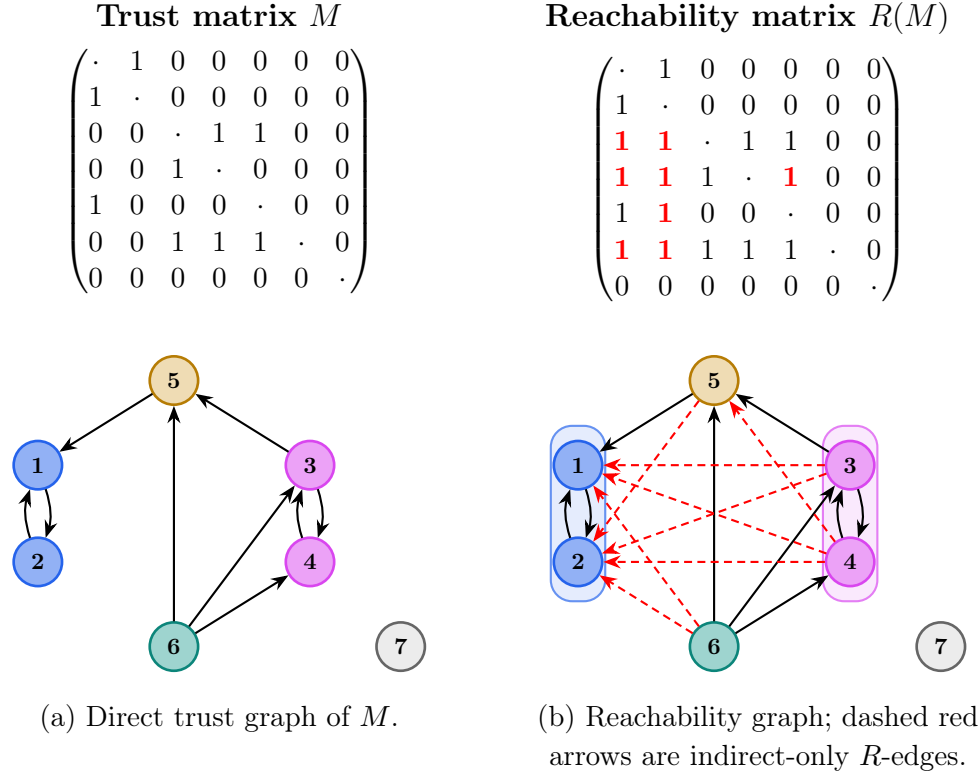
\begin{figure}[!htb]
\centering
\begin{minipage}[t]{0.48\textwidth}
\centering
\textbf{Trust matrix $M$}\\[0.5em]
{\small\(\begin{pmatrix}
\cdot & 1 & 0 & 0 & 0 & 0 & 0 \\
1 & \cdot & 0 & 0 & 0 & 0 & 0 \\
0 & 0 & \cdot & 1 & 1 & 0 & 0 \\
0 & 0 & 1 & \cdot & 0 & 0 & 0 \\
1 & 0 & 0 & 0 & \cdot & 0 & 0 \\
0 & 0 & 1 & 1 & 1 & \cdot & 0 \\
0 & 0 & 0 & 0 & 0 & 0 & \cdot
\end{pmatrix}\)}
\end{minipage}\hfill
\begin{minipage}[t]{0.48\textwidth}
\centering
\textbf{Reachability matrix $\Rmat(M)$}\\[0.5em]
{\small\(\begin{pmatrix}
\cdot & 1 & 0 & 0 & 0 & 0 & 0 \\
1 & \cdot & 0 & 0 & 0 & 0 & 0 \\
\textcolor{red}{\mathbf{1}} & \textcolor{red}{\mathbf{1}} & \cdot & 1 & 1 & 0 & 0 \\
\textcolor{red}{\mathbf{1}} & \textcolor{red}{\mathbf{1}} & 1 & \cdot & \textcolor{red}{\mathbf{1}} & 0 & 0 \\
1 & \textcolor{red}{\mathbf{1}} & 0 & 0 & \cdot & 0 & 0 \\
\textcolor{red}{\mathbf{1}} & \textcolor{red}{\mathbf{1}} & 1 & 1 & 1 & \cdot & 0 \\
0 & 0 & 0 & 0 & 0 & 0 & \cdot
\end{pmatrix}\)}
\end{minipage}
\vspace{1.5em}

\begin{minipage}[t]{0.48\textwidth}
\centering
\begin{tikzpicture}[
  scale=0.8, transform shape,
  >=Stealth, thick,
  every node/.style={circle, draw, minimum size=8mm, font=\small\bfseries, inner sep=0pt, line width=1pt}
]
\node[fill=factionBlue!50, draw=factionBlue] (n1) at (0, 0.8) {1};
\node[fill=factionBlue!50, draw=factionBlue] (n2) at (0, -0.8) {2};
\node[fill=factionGold!30, draw=factionGold!90!black] (n5) at (2.25, 2.2) {5};
\node[fill=factionMagenta!50, draw=factionMagenta] (n3) at (4.5, 0.8) {3};
\node[fill=factionMagenta!50, draw=factionMagenta] (n4) at (4.5, -0.8) {4};
\node[fill=factionTeal!40, draw=factionTeal!90!black] (n6) at (2.25, -2.2) {6};
\node[fill=gray!15, draw=gray!80!black] (n7) at (6.0, -2.2) {7};

\draw[->] (n1) to[bend left=18] (n2);
\draw[->] (n2) to[bend left=18] (n1);
\draw[->] (n3) to[bend left=18] (n4);
\draw[->] (n4) to[bend left=18] (n3);
\draw[->] (n5) -- (n1);
\draw[->] (n3) -- (n5);
\draw[->] (n6) -- (n3);
\draw[->] (n6) -- (n4);
\draw[->] (n6) -- (n5);
\end{tikzpicture}
\\[0.5em]
{\small (a) Direct trust graph of $M$.}
\end{minipage}\hfill
\begin{minipage}[t]{0.48\textwidth}
\centering
\begin{tikzpicture}[
  scale=0.8, transform shape,
  >=Stealth, thick,
  every node/.style={circle, draw, minimum size=8mm, font=\small\bfseries, inner sep=0pt, line width=1pt}
]
\fill[factionBlue!12, rounded corners=8pt]
  (-0.45, -1.45) rectangle (0.45, 1.45);
\draw[draw=factionBlue!70, line width=0.8pt, rounded corners=8pt]
  (-0.45, -1.45) rectangle (0.45, 1.45);
\fill[factionMagenta!12, rounded corners=8pt]
  (4.05, -1.45) rectangle (4.95, 1.45);
\draw[draw=factionMagenta!70, line width=0.8pt, rounded corners=8pt]
  (4.05, -1.45) rectangle (4.95, 1.45);

\node[fill=factionBlue!50, draw=factionBlue] (n1) at (0, 0.8) {1};
\node[fill=factionBlue!50, draw=factionBlue] (n2) at (0, -0.8) {2};
\node[fill=factionGold!30, draw=factionGold!90!black] (n5) at (2.25, 2.2) {5};
\node[fill=factionMagenta!50, draw=factionMagenta] (n3) at (4.5, 0.8) {3};
\node[fill=factionMagenta!50, draw=factionMagenta] (n4) at (4.5, -0.8) {4};
\node[fill=factionTeal!40, draw=factionTeal!90!black] (n6) at (2.25, -2.2) {6};
\node[fill=gray!15, draw=gray!80!black] (n7) at (6.0, -2.2) {7};

\draw[->] (n1) to[bend left=18] (n2);
\draw[->] (n2) to[bend left=18] (n1);
\draw[->] (n3) to[bend left=18] (n4);
\draw[->] (n4) to[bend left=18] (n3);
\draw[->] (n5) -- (n1);
\draw[->] (n3) -- (n5);
\draw[->] (n6) -- (n3);
\draw[->] (n6) -- (n4);
\draw[->] (n6) -- (n5);

\draw[->, densely dashed, red, thick] (n3) -- (n1);
\draw[->, densely dashed, red, thick] (n3) -- (n2);
\draw[->, densely dashed, red, thick] (n4) -- (n1);
\draw[->, densely dashed, red, thick] (n4) -- (n2);
\draw[->, densely dashed, red, thick] (n4) -- (n5);
\draw[->, densely dashed, red, thick] (n5) -- (n2);
\draw[->, densely dashed, red, thick] (n6) -- (n1);
\draw[->, densely dashed, red, thick] (n6) -- (n2);
\end{tikzpicture}
\\[0.5em]
{\small (b) Reachability graph; dashed red arrows are indirect-only $\Rmat$-edges.}
\end{minipage}
\caption{The $n = 7$ running example. Top: the trust matrix $M$ and its reachability matrix $\Rmat(M)$; the entries highlighted in red are the eight indirect reachabilities, so $|R(M)| = 17$. Bottom: (a) the trust graph and (b) the reachability graph, with indirect-only edges dashed red and the clusters shaded. Here $\{1, 2\}$ is inward and $\{3, 4\}$ is outward (since $3$ trusts out of the cluster by trusting $5$); $5$, $6$, and $7$ are lone agents. Later figures revisit this example as the proof progresses.}
\label{fig:example}
\end{figure}


\section{The descent measure}\label{sec:basic}

\begin{lemma}[$0 \to 1$ moves are flat]\label{lem:01-flat}
Every $0 \to 1$ move preserves $\Rmat$, and hence is flat.
\end{lemma}

\begin{proof}
Suppose $\gos(M, a(y), z)$ is a $0 \to 1$ move, so $z(a) = a(y) = 1$ pre-state, and the move sets $z(y) := 1$. The pre-state contains the length-$2$ path $z \to a \to y$. Any post-state path that uses the new edge $(z, y)$ can be rewritten by replacing that edge with the segment $z \to a \to y$; since the move modified no entry of $M$ other than $z(y)$, the rest of the path is unchanged and lies in the pre-state. Hence every post-state reachability already held in the pre-state, and since the move removes no edges, $\Rmat(M)$ is unchanged.
\end{proof}

\begin{corollary}[$|R(M)|$ is monotone non-increasing]\label{cor:R-monotone}
For every firing move $M \to M'$, $|R(M)| \geq |R(M')|$.
\end{corollary}

\begin{proof}
By Lemma~\ref{lem:01-flat}, $0 \to 1$ moves preserve $\Rmat$, so $|R(M)| = |R(M')|$. A $1 \to 0$ move removes one direct trust edge, $z(y)$, from $M$. Removing an edge from a directed graph can only shrink the reachability relation: every path in the post-move graph is a path in the pre-move graph. So $\Rmat(M) \supseteq \Rmat(M')$, hence $|R(M)| \geq |R(M')|$.
\end{proof}

From the above, any decrease in $|R(M)|$ comes from a $1 \to 0$ move, and only when removing the direct edge $z(y)$ also removes some pair from $\Rmat$.

\begin{lemma}[Reachability invariance]\label{lem:R}
Any run of flat moves from $M$ leaves $\Rmat$ unchanged.
\end{lemma}

\begin{proof}
Let $M \to M'$ be a flat firing move. By Corollary~\ref{cor:R-monotone}, $\Rmat(M) \supseteq \Rmat(M')$; flatness gives $|R(M)| = |R(M')|$; a superset of equal size forces $\Rmat(M) = \Rmat(M')$. Applying this along a flat-move sequence keeps $\Rmat$ constant.
\end{proof}

\begin{corollary}\label{cor:clusters}
Any run of flat moves leaves the R-clusters unchanged.
\end{corollary}

\begin{lemma}[Paths stay in clusters]\label{lem:scc-path}
Let $u$ and $v$ lie in a common R-cluster $C$. Then every directed $M$-path from $u$ to $v$ has all its vertices in $C$.
\end{lemma}

\begin{proof}
Let $x$ be any vertex on the path. $u$ reaches $x$ in $\Rmat$, and $x$ reaches $v$ in $\Rmat$. Since $u$ and $v$ lie in the same cluster, $v$ reaches $u$ in $\Rmat$. Composing: $x$ reaches $v$ reaches $u$, so $x$ reaches $u$. Combined with $u$ reaches $x$, the vertex $x$ is mutually reachable with $u$, so $x \in C$.
\end{proof}


\section{The two kinds of cluster}\label{sec:dichotomy}

On a flat-closure the R-cluster structure is fixed (Corollary~\ref{cor:clusters}). We define two kinds of clusters.

\begin{definition}[Inward and outward clusters]\label{def:inout}
A cluster $C$ is \emph{outward} if some outsider is reachable from $C$ in $\Rmat$ (that is, $\Rmat_{uy} = 1$ for some $u \in C$ and $y \notin C$), and \emph{inward} otherwise.
\end{definition}

For $C$ to be outward, some member must \emph{directly} trust an outsider: any path leaving $C$ takes a first step out, so if no member trusts an outsider, no member can reach one. And since $C$ is strongly connected, an outsider reachable from one member is reachable from all, so being outward is a property of the \emph{cluster}, not of any single member. The two are treated differently by this proof: in our sequence, an outward cluster gets disassembled so that its members no longer trust each other, while an inward one gets completed so that all its members directly trust each other. Figure~\ref{fig:cases} shows a minimal $n = 4$ instance of each, and we will now describe the process in detail.

\begin{figure}[!htb]
\centering
\begin{minipage}[t]{0.48\textwidth}
\centering
\begin{tikzpicture}[
  >=Stealth, thick,
  every node/.style={circle, draw, minimum size=8mm, font=\small\bfseries, inner sep=0pt, line width=1pt}
]
\fill[factionBlue!12, draw=factionBlue!70, line width=0.8pt, rounded corners=8pt]
  (-0.65, -1.55) rectangle (0.65, 1.55);
\draw[draw=factionBlue!70, line width=0.8pt, rounded corners=8pt]
  (-0.65, -1.55) rectangle (0.65, 1.55);

\node[fill=factionBlue!50, draw=factionBlue] (n1) at (0, 0.9) {1};
\node[fill=factionBlue!50, draw=factionBlue] (n2) at (0, -0.9) {2};
\node[fill=gray!15, draw=gray!80!black] (n3) at (2.4, 0.9) {3};
\node[fill=gray!15, draw=gray!80!black] (n4) at (2.4, -0.9) {4};

\draw[->] (n1) to[bend left=18] (n2);
\draw[->] (n2) to[bend left=18] (n1);
\end{tikzpicture}\\[1em]
{\small (a) Inward: $C = \{1, 2\}$ has no trust edges to outsiders.}
\end{minipage}\hfill
\begin{minipage}[t]{0.48\textwidth}
\centering
\begin{tikzpicture}[
  >=Stealth, thick,
  every node/.style={circle, draw, minimum size=8mm, font=\small\bfseries, inner sep=0pt, line width=1pt}
]
\fill[factionBlue!12, draw=factionBlue!70, line width=0.8pt, rounded corners=8pt]
  (-0.65, -1.55) rectangle (0.65, 1.55);
\draw[draw=factionBlue!70, line width=0.8pt, rounded corners=8pt]
  (-0.65, -1.55) rectangle (0.65, 1.55);

\node[fill=factionBlue!50, draw=factionBlue] (n1) at (0, 0.9) {1};
\node[fill=factionBlue!50, draw=factionBlue] (n2) at (0, -0.9) {2};
\node[fill=factionGold!30, draw=factionGold!90!black, line width=1.5pt] (n3) at (2.1, 0) {$w$};
\node[fill=gray!15, draw=gray!80!black] (n4) at (3.9, 0) {4};

\draw[->] (n1) to[bend left=18] (n2);
\draw[->] (n2) to[bend left=18] (n1);
\draw[->, thick, factionGold!90!black] (n1) -- (n3);
\end{tikzpicture}\\[1em]
{\small (b) Outward: at least one member of $C = \{1, 2\}$ trusts the outsider $w$.}
\end{minipage}
\caption{The two kinds of cluster.}
\label{fig:cases}
\end{figure}
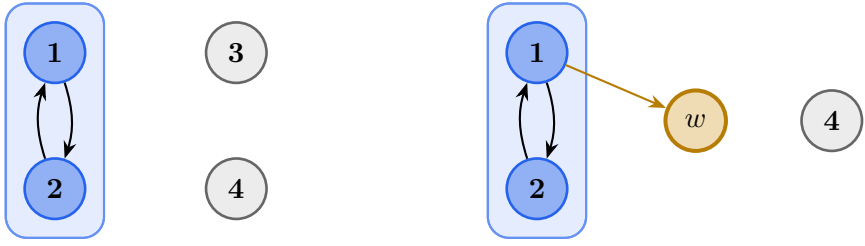


\section{Outward clusters yield progress}\label{sec:outward}

\begin{definition}[Common trustee]\label{def:trustee}
An outsider $w \notin C$ is a \emph{common trustee} of $C$ when every member of $C$ trusts $w$. A common trustee trusts no member of $C$ in return: otherwise $w$ would reach all of $C$ and be reached by all of $C$, putting $w$ into the cluster by definition.
\end{definition}

We dissolve an outward cluster by manufacturing a common trustee and using it as a speaker who causes all the agents in the cluster to distrust each other; this strictly lowers $|R(M)|$. An outward cluster may have a common trustee at the start, but if one does not exist, it can be manufactured out of a reachable outsider with a short run of flat moves, as we will show here.

\begin{lemma}[Common trustee]\label{lem:trustee}
Let $C$ be an outward cluster. A finite sequence of flat $0 \to 1$ moves makes some outsider $w \notin C$ a common trustee of $C$ (so $z(w) = 1$ for every $z \in C$).
\end{lemma}

\begin{proof}
Since $C$ is outward, some member directly trusts an outsider; fix one such outsider $w$, so $a(w) = 1$ for some $a \in C$. Next, we spread that trust of $w$ to the rest of the cluster. For any $z \in C$ with $z(w) = 0$, strong connectivity gives an $M$-path from $z$ to $a$, which stays inside $C$ by Lemma~\ref{lem:scc-path}. Walk the opinion of $w$ backward along this path: wherever consecutive vertices $z_i \to z_{i+1}$ have $z_{i+1}(w) = 1$ and $z_i(w) = 0$, fire $\gos(M, z_{i+1}(w), z_i)$, a $0 \to 1$ move that sets $z_i(w) := 1$ and is flat by Lemma~\ref{lem:01-flat}. Repeating over all members brings every $z \in C$ to $z(w) = 1$. Figure~\ref{fig:agreement} shows the backward walk in general, and Figure~\ref{fig:ex-sa} shows it on the running example.
\end{proof}

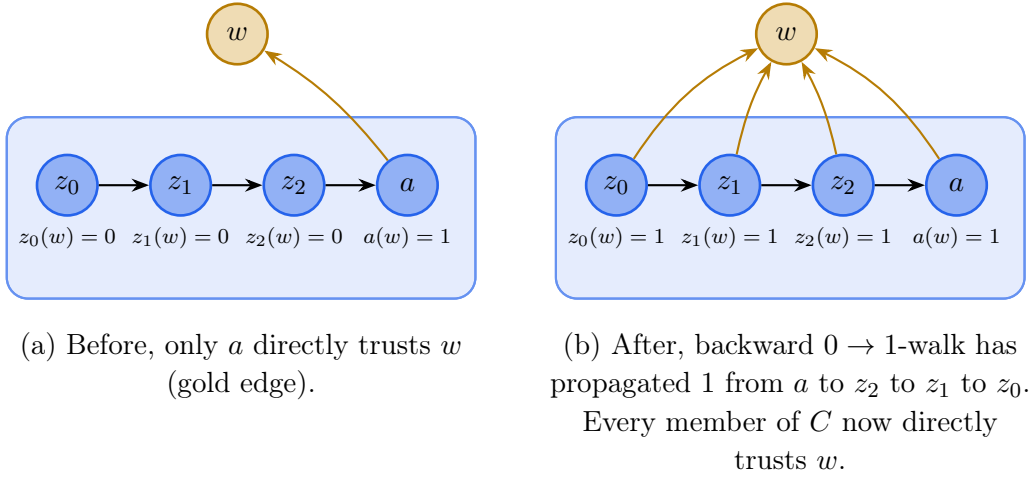
\begin{figure}[!htb]
\centering
\begin{minipage}[t]{0.47\textwidth}
\centering
\begin{tikzpicture}[
  >=Stealth, thick,
  every node/.style={circle, draw, minimum size=8mm, font=\small\bfseries, inner sep=0pt, line width=1pt}
]
\fill[factionBlue!12, draw=factionBlue!70, line width=0.8pt, rounded corners=8pt]
  (-0.8, -1.5) rectangle (5.4, 0.9);
\draw[draw=factionBlue!70, line width=0.8pt, rounded corners=8pt]
  (-0.8, -1.5) rectangle (5.4, 0.9);

\node[fill=factionBlue!50, draw=factionBlue] (z0) at (0, 0) {$z_0$};
\node[fill=factionBlue!50, draw=factionBlue] (z1) at (1.5, 0) {$z_1$};
\node[fill=factionBlue!50, draw=factionBlue] (z2) at (3.0, 0) {$z_2$};
\node[fill=factionBlue!50, draw=factionBlue] (zk) at (4.5, 0) {$a$};

\node[fill=factionGold!30, draw=factionGold!90!black] (y) at (2.25, 2.0) {$w$};

\draw[->] (z0) -- (z1);
\draw[->] (z1) -- (z2);
\draw[->] (z2) -- (zk);

\draw[->, factionGold!90!black, thick] (zk) to[bend right=10] (y);

\node[font=\scriptsize, draw=none] at (0, -0.7) {$z_0(w)=0$};
\node[font=\scriptsize, draw=none] at (1.5, -0.7) {$z_1(w)=0$};
\node[font=\scriptsize, draw=none] at (3.0, -0.7) {$z_2(w)=0$};
\node[font=\scriptsize, draw=none] at (4.5, -0.7) {$a(w)=1$};
\end{tikzpicture}
\\[0.5em]
{\small (a) Before, only $a$ directly trusts $w$ (gold edge).}
\end{minipage}\hfill
\begin{minipage}[t]{0.47\textwidth}
\centering
\begin{tikzpicture}[
  >=Stealth, thick,
  every node/.style={circle, draw, minimum size=8mm, font=\small\bfseries, inner sep=0pt, line width=1pt}
]
\fill[factionBlue!12, draw=factionBlue!70, line width=0.8pt, rounded corners=8pt]
  (-0.8, -1.5) rectangle (5.4, 0.9);
\draw[draw=factionBlue!70, line width=0.8pt, rounded corners=8pt]
  (-0.8, -1.5) rectangle (5.4, 0.9);

\node[fill=factionBlue!50, draw=factionBlue] (z0) at (0, 0) {$z_0$};
\node[fill=factionBlue!50, draw=factionBlue] (z1) at (1.5, 0) {$z_1$};
\node[fill=factionBlue!50, draw=factionBlue] (z2) at (3.0, 0) {$z_2$};
\node[fill=factionBlue!50, draw=factionBlue] (zk) at (4.5, 0) {$a$};

\node[fill=factionGold!30, draw=factionGold!90!black] (y) at (2.25, 2.0) {$w$};

\draw[->] (z0) -- (z1);
\draw[->] (z1) -- (z2);
\draw[->] (z2) -- (zk);

\draw[->, factionGold!90!black, thick] (z0) to[bend left=15] (y);
\draw[->, factionGold!90!black, thick] (z1) to[bend left=8] (y);
\draw[->, factionGold!90!black, thick] (z2) to[bend right=8] (y);
\draw[->, factionGold!90!black, thick] (zk) to[bend right=15] (y);

\node[font=\scriptsize, draw=none] at (0, -0.7) {$z_0(w)=1$};
\node[font=\scriptsize, draw=none] at (1.5, -0.7) {$z_1(w)=1$};
\node[font=\scriptsize, draw=none] at (3.0, -0.7) {$z_2(w)=1$};
\node[font=\scriptsize, draw=none] at (4.5, -0.7) {$a(w)=1$};
\end{tikzpicture}
\\[0.5em]
{\small (b) After, backward $0 \to 1$-walk has propagated $1$ from $a$ to $z_2$ to $z_1$ to $z_0$. Every member of $C$ now directly trusts $w$.}
\end{minipage}
\caption{The common-trustee process (Lemma~\ref{lem:trustee}) before and after the backwards opinion walk. Each backward step is a $0 \to 1$ move and therefore flat, so $\Rmat$ never changes.}
\label{fig:agreement}
\end{figure}

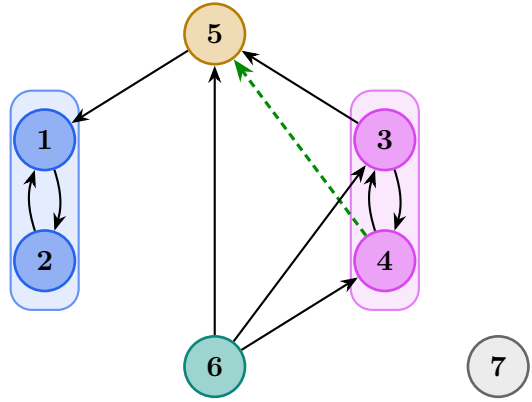
\begin{figure}[!htb]
\centering
\begin{tikzpicture}[
  >=Stealth, thick,
  every node/.style={circle, draw, minimum size=8mm, font=\small\bfseries, inner sep=0pt, line width=1pt}
]
\fill[factionBlue!12, rounded corners=8pt]
  (-0.45, -1.45) rectangle (0.45, 1.45);
\draw[draw=factionBlue!70, line width=0.8pt, rounded corners=8pt]
  (-0.45, -1.45) rectangle (0.45, 1.45);
\fill[factionMagenta!12, rounded corners=8pt]
  (4.05, -1.45) rectangle (4.95, 1.45);
\draw[draw=factionMagenta!70, line width=0.8pt, rounded corners=8pt]
  (4.05, -1.45) rectangle (4.95, 1.45);

\node[fill=factionBlue!50, draw=factionBlue] (n1) at (0, 0.8) {1};
\node[fill=factionBlue!50, draw=factionBlue] (n2) at (0, -0.8) {2};
\node[fill=factionGold!30, draw=factionGold!90!black] (n5) at (2.25, 2.2) {5};
\node[fill=factionMagenta!50, draw=factionMagenta] (n3) at (4.5, 0.8) {3};
\node[fill=factionMagenta!50, draw=factionMagenta] (n4) at (4.5, -0.8) {4};
\node[fill=factionTeal!40, draw=factionTeal!90!black] (n6) at (2.25, -2.2) {6};
\node[fill=gray!15, draw=gray!80!black] (n7) at (6.0, -2.2) {7};
\draw[->] (n1) to[bend left=18] (n2);
\draw[->] (n2) to[bend left=18] (n1);
\draw[->] (n3) to[bend left=18] (n4);
\draw[->] (n4) to[bend left=18] (n3);
\draw[->] (n5) -- (n1);
\draw[->] (n3) -- (n5);
\draw[->, green!55!black, very thick, densely dashed] (n4) -- (n5);
\draw[->] (n6) -- (n3);
\draw[->] (n6) -- (n4);
\draw[->] (n6) -- (n5);
\end{tikzpicture}
\caption{The common-trustee step. Cluster $\{3, 4\}$ is outward, and a single move propagates 3's trust of 5 across the edge $4 \to 3$, setting $4(5) := 1$ (green dashed), so both members now trust $5$.}
\label{fig:ex-sa}
\end{figure}

\begin{lemma}[Dissolution]\label{lem:B}
Let $w$ be a common trustee of a cluster $C$. Then a finite sequence of firing moves brings $M$ to a state in which $z(y) = 0$ for all distinct $z, y \in C$ (no-one in the cluster trusts anyone else in the cluster anymore and it is said to be dissolved), and $|R(M)|$ strictly decreases by at least $|C|(|C|-1) \geq 2$.
\end{lemma}

\begin{proof}
A common trustee trusts no member of $C$, so $w(y) = 0$ for every $y \in C$. We use $w$ as the speaker: while some distinct $z, y \in C$ have $z(y) = 1$, apply $\gos(M, w(y), z)$, which sets $z(y) := 0$. Each pass deletes a distinct intra-cluster edge, so after finitely many passes $M$ is zero on $C \times C$, i.e. $C$ is dissolved.

\emph{Strict drop.} Before the process $C$ is a cluster, so $\Rmat_{zy} = 1$ for all distinct $z, y \in C$, worth $|C|(|C|-1)$ to $|R(M)|$. By Lemma~\ref{lem:scc-path}, every $M$-path between two members of $C$ has all its edges inside $C \times C$; once those edges are gone, no path between two members survives, so every such $\Rmat_{zy}$ falls to $0$. Hence $|R(M)|$ drops by at least $|C|(|C|-1) \geq 2$. This is only a lower bound: an outside path routed through $C$ is also severed when the cluster is dissolved, and its two endpoints need not both lie in $C$, so such a loss adds to the drop without appearing in $|C|(|C|-1)$. Figure~\ref{fig:ex-lemmaB} shows the dissolution on the running example.
\end{proof}

\begin{figure}[!htb]
\centering
\begin{tikzpicture}[
  >=Stealth, thick,
  every node/.style={circle, draw, minimum size=8mm, font=\small\bfseries, inner sep=0pt, line width=1pt}
]
\fill[factionBlue!12, rounded corners=8pt]
  (-0.45, -1.45) rectangle (0.45, 1.45);
\draw[draw=factionBlue!70, line width=0.8pt, rounded corners=8pt]
  (-0.45, -1.45) rectangle (0.45, 1.45);

\node[fill=factionBlue!50, draw=factionBlue] (n1) at (0, 0.8) {1};
\node[fill=factionBlue!50, draw=factionBlue] (n2) at (0, -0.8) {2};
\node[fill=factionGold!30, draw=factionGold!90!black] (n5) at (2.25, 2.2) {5};
\node[fill=factionMagenta!50, draw=factionMagenta] (n3) at (4.5, 0.8) {3};
\node[fill=factionMagenta!50, draw=factionMagenta] (n4) at (4.5, -0.8) {4};
\node[fill=factionTeal!40, draw=factionTeal!90!black] (n6) at (2.25, -2.2) {6};
\node[fill=gray!15, draw=gray!80!black] (n7) at (6.0, -2.2) {7};
\draw[->] (n1) to[bend left=18] (n2);
\draw[->] (n2) to[bend left=18] (n1);
\draw[->, red!75!black, very thick, densely dashed] (n3) to[bend left=18] (n4);
\draw[->, red!75!black, very thick, densely dashed] (n4) to[bend left=18] (n3);
\draw[->] (n5) -- (n1);
\draw[->] (n3) -- (n5);
\draw[->] (n4) -- (n5);
\draw[->] (n6) -- (n3);
\draw[->] (n6) -- (n4);
\draw[->] (n6) -- (n5);
\end{tikzpicture}
\caption{Dissolution. The common trustee $5$ acts as speaker for two moves, $\gos(M, 5(4), 3)$ and $\gos(M, 5(3), 4)$, removing the edges $3(4)$ and $4(3)$ (red dashed) and dissolving the cluster $\{3, 4\}$; every other edge is unchanged.}
\label{fig:ex-lemmaB}
\end{figure}
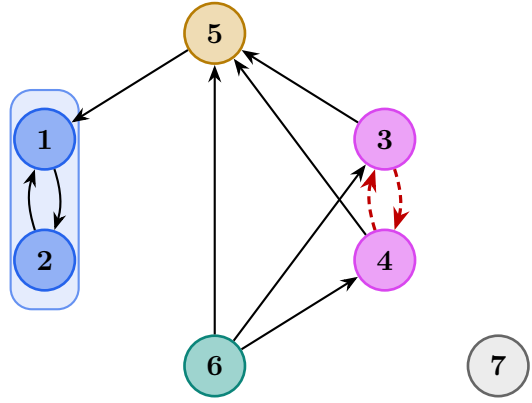


\section{Inward clusters lock into factions}\label{sec:post-processing}

An inward cluster cannot be dissolved in this way since it trusts no outsider who could be turned into a common trustee. Instead, we first \emph{saturate} the cluster until its members all trust each other, then \emph{align} every outsider who trusts into the cluster by overwriting its opinions to match the cluster's. The result is a structure that the companion paper~\cite{absorbing} calls a \emph{faction}: a cluster of directly mutually-trusting agents (the \emph{core}) who only trust each other, together with the outsiders that directly trust this cluster and no-one else (the \emph{periphery}). We use $C$ for a core, $P(C)$ for its periphery, and $F(C) = C \cup P(C)$ for the faction. (The companion paper uses slightly different notation: for faction index $\ell$, there are faction $F_\ell$, core $C_\ell$, and peripherals $F_\ell \setminus C_\ell$.)

\begin{definition}[Periphery and anchor]\label{def:anchor}
For an inward cluster $C$, the \emph{periphery} of $C$, written $P(C)$, consists of the agents outside $C$ that directly trust some member of $C$. We read $P(C)$ dynamically: it is evaluated when $C$ is handled in the fixed order of the saturate--align process below. An agent that an earlier faction already captured has had its row overwritten to trust only that earlier core, so it no longer trusts into $C$ and is not a candidate. For each peripheral $z \in P(C)$, we fix an \emph{anchor} $a_z \in C$ with $z(a_z) = 1$.
\end{definition}

\begin{definition}[Aligned peripheral]\label{def:alignment}
A peripheral $z \in P(C)$ is \emph{aligned with} its anchor $a_z$ when row $z$ agrees with row $a_z$ everywhere off $z(a_z)$ and $z(z)$; that is, $z(y) = a_z(y)$ for every $y \notin \{a_z, z\}$.
\end{definition}

A peripheral already trusts its anchor, but its other opinions need not match $a_z$'s yet. The \emph{alignment sequence for $z$} brings them into line: while some $y \notin \{a_z, z\}$ has $z(y) \neq a_z(y)$, fire $\gos(M, a_z(y), z)$, which sets $z(y) := a_z(y)$. Each move edits one such $z(y)$ and introduces no new disagreement, so after finitely many moves $z$ is aligned. The firing condition $z(a_z) = 1$ holds throughout, since $a_z$ is never a target if it is itself the speaker.

\paragraph{The saturate--align process.}
For each inward cluster $C$ in some fixed order:
\begin{enumerate}
\item \emph{Saturate} $C$: apply intra-$C$ $0 \to 1$ firing moves greedily until $M$ on $C \times C$ equals $K_C$ (the complete digraph on $C$).
\item \emph{Align}: for each $z \in P(C)$, run the alignment sequence for $z$.
\end{enumerate}
Figure~\ref{fig:ex-lemmaA} shows the process on the running example.

\begin{definition}[Silent and active agents]\label{def:silent}
After the saturate--align process, an agent is \emph{silent} if it belongs to one of the factions $F(C)$ built by the process, and \emph{active} otherwise. These labels are assigned only where the saturate--align process runs, namely when every cluster is inward; an outward cluster is dissolved instead (Section~\ref{sec:outward}) and its members are never sorted into silent and active.
\end{definition}

\begin{lemma}[Structure after saturate--align]\label{lem:post-processing}
Suppose every cluster of $M$ is inward, and run the saturate--align process on every cluster. Afterwards:
\begin{itemize}
\item[(a)] \emph{(Factions lock.)} For every cluster $C$, each core agent ends up trusting exactly the other core agents and each peripheral trusts exactly the core; no member of $F(C)$ can ever be a listener afterward.
\item[(b)] \emph{(Factions are disjoint.)} $F(C) \cap F(C') = \emptyset$ whenever $C \neq C'$.
\end{itemize}
\end{lemma}

\begin{proof}
\emph{(a) Faction structure.} Saturation applies intra-$C$ $0 \to 1$ moves, each flat (Lemma~\ref{lem:01-flat}) and adding a $1$ to $C \times C$, so no edge is added twice and the process halts within $|C|(|C|-1)$ moves. It halts only at $K_C$: since $C$ is an R-cluster, $\Rmat$ restricted to $C$ is complete by definition, and by Lemma~\ref{lem:scc-path} every $M$-path between two members of $C$ stays inside $C$, so the transitive closure of $M$ restricted to $C \times C$ agrees with $\Rmat$ restricted to $C$ --- also complete. So whenever $M$ on $C \times C$ is not yet $K_C$, it is not yet transitively closed: some triple $p, r, q \in C$ has $p(r) = 1$, $r(q) = 1$, $p(q) = 0$, which is exactly a legal $0 \to 1$ move (speaker $r$, listener $p$, target $q$). The process can always proceed until $M$ on $C \times C$ is $K_C$, at which point trust inside $C$ is transitive by construction.

\emph{$C$ stays inward.} No move makes a $C$-member trust an outsider. A $1 \to 0$ move only removes a trust edge. A $0 \to 1$ move makes a listener $z \in C$ trust some new $y$, copied from its speaker $a$: since $z$ trusts $a$, $a \in C$, and $a$ trusts only $C$, so $y \in C$. Handling another cluster $C'$ never touches a $C$-row: saturating $C'$ will only edit $C'$-rows, and aligning $C'$ will only edit $P(C')$-rows, which contain no $C$-member, since members of an inward $C$ trust no outsiders.

\emph{Peripherals copy the core.} Consider a peripheral $z \in P(C)$ with anchor $a_z \in C$. By the two paragraphs above, $a_z$ trusts exactly $C \setminus \{a_z\}$. Aligning $z$ to $a_z$ overwrites row $z$ with row $a_z$ on every column except $z(a_z)$ and $z(z)$; with the anchor edge $z \to a_z$ still in place, $z$ ends trusting exactly $C$. Distinct peripherals are defined by disjoint rows, so the alignments are independent and their order does not matter.

\emph{No more firing moves are possible.} A faction member trusts only core members, and core members and peripherals alike trust exactly the core. So a faction member and any speaker it trusts agree on every target: each trusts a third agent precisely when it lies in the core, and the firing condition (a disagreement) never arises.

\emph{(b) Factions are disjoint.} Distinct clusters are disjoint, and a core member of $C$ trusts only $C$ by part~(a), so it cannot also be peripheral. Nor can an agent be peripheral to two different cores: suppose $z \in P(C)$ and $z \in P(C')$. Aligning $z$ to $C$ leaves its row trusting only $C$, so by the time $C'$ is aligned $z$ trusts no member of $C'$ and has left $P(C')$. Hence each agent belongs to exactly one faction.
\end{proof}

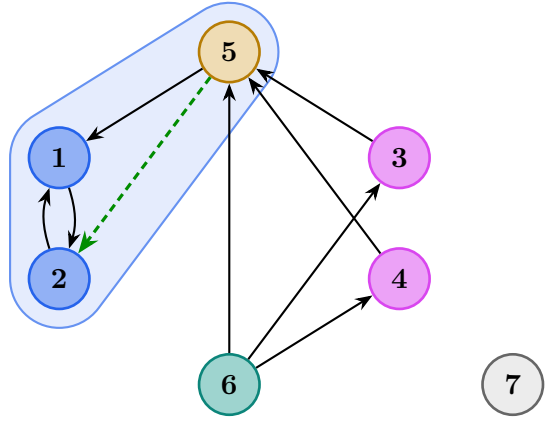
\begin{figure}[!htb]
\centering
\begin{tikzpicture}[
  >=Stealth, thick,
  every node/.style={circle, draw, minimum size=8mm, font=\small\bfseries, inner sep=0pt, line width=1pt}
]
\fill[factionBlue!12]
  (2.770, 1.810)
  arc (-36.87:121.89:0.65) -- (-0.343, 1.352)
  arc (121.89:180.00:0.65) -- (-0.650, -0.800)
  arc (-180.00:-36.87:0.65) -- cycle;
\draw[factionBlue!70, line width=0.8pt]
  (2.770, 1.810)
  arc (-36.87:121.89:0.65) -- (-0.343, 1.352)
  arc (121.89:180.00:0.65) -- (-0.650, -0.800)
  arc (-180.00:-36.87:0.65) -- cycle;

\node[fill=factionBlue!50, draw=factionBlue] (n1) at (0, 0.8) {1};
\node[fill=factionBlue!50, draw=factionBlue] (n2) at (0, -0.8) {2};
\node[fill=factionGold!30, draw=factionGold!90!black] (n5) at (2.25, 2.2) {5};
\node[fill=factionMagenta!50, draw=factionMagenta] (n3) at (4.5, 0.8) {3};
\node[fill=factionMagenta!50, draw=factionMagenta] (n4) at (4.5, -0.8) {4};
\node[fill=factionTeal!40, draw=factionTeal!90!black] (n6) at (2.25, -2.2) {6};
\node[fill=gray!15, draw=gray!80!black] (n7) at (6.0, -2.2) {7};
\draw[->] (n1) to[bend left=18] (n2);
\draw[->] (n2) to[bend left=18] (n1);
\draw[->] (n5) -- (n1);
\draw[->, green!55!black, very thick, densely dashed] (n5) -- (n2);
\draw[->] (n3) -- (n5);
\draw[->] (n4) -- (n5);
\draw[->] (n6) -- (n3);
\draw[->] (n6) -- (n4);
\draw[->] (n6) -- (n5);
\end{tikzpicture}
\caption{Saturate--align process on the running example. \{1, 2\} happens to already be saturated; agent $5$ has anchor $1$. The alignment sequence for $5$ is a single move setting $5(2) := 1$ (green dashed). After alignment, row $5$ equals row $1$ off the diagonal positions, and $5$ joins the faction as a peripheral member; $F(C) = \{1, 2, 5\}$ becomes silent.}
\label{fig:ex-lemmaA}
\end{figure}

\section{Zipping the active agents into factions}\label{sec:zipping}

After the saturate--align process, the formerly inward factions are locked: their members are silent and can never be listeners again. Each remaining \emph{active} agent will now be \emph{zipped} into a faction it reaches, joining that faction's periphery, as described below.

\begin{definition}[Singleton core; generalized faction]\label{def:single-core}
An agent that trusts no one forms a core by itself, a \emph{singleton core}. From here on a \emph{core} is either a multi-agent cluster locked by Lemma~\ref{lem:post-processing} or a singleton core; either kind trusts nothing outside itself. This matches the companion paper's factions, whose core may be of any size~\cite{absorbing}.
\end{definition}

An active agent never directly trusts a multi-agent core. An active agent's row is never edited during the process, so if it trusts a core member now, it trusted that member already when the cluster was aligned; it was therefore in the periphery at that point (Definition~\ref{def:anchor}), underwent alignment, and is silent --- a contradiction. So an active agent only trusts peripherals or other active agents.

\begin{lemma}[Active agents reach a core]\label{lem:reach-core}
After the process, every active agent reaches a core member of a faction in $\Rmat$.
\end{lemma}

\begin{proof}
Every cluster was silenced into a faction, so no active agent lies in one; two active agents that were mutually reachable would share a cluster, so distinct active agents are not mutually reachable and form a DAG in $\Rmat$. Starting from any active agent $z$, follow trust edges through active agents; the DAG has no cycle, so the walk halts at an active agent $a$ that trusts no other active agent. By the remark above $a$ trusts only peripherals, or no one. If no one, $a$ is itself a singleton core. Otherwise $a$ trusts a peripheral of a faction with core $C$, and that peripheral trusts the members of $C$, so active agent $z$ reaches a core member of $C$ in $R$.
\end{proof}

\begin{lemma}[Zipping]\label{lem:zipping}
After the saturate--align process (Lemma~\ref{lem:post-processing}), either
\begin{itemize}
\item[(a)] every active agent trusts no one, or trusts exactly one agent that itself trusts no one; then $M$ is absorbing; or
\item[(b)] otherwise, some active agent trusts multiple agents, and a finite sequence of firing moves makes this active agent a peripheral of a core and strictly decreases $|R(M)|$.
\end{itemize}
\end{lemma}

\begin{proof}
\emph{(a)} An active agent that trusts no one is a singleton core; one that trusts a single trustless agent $c$ is a peripheral of the singleton core $\{c\}$. Together with the silent agents, each trusting exactly its core by Lemma~\ref{lem:post-processing}(a), every agent now trusts exactly the core of its faction (no-one, for a singleton core). Such a state is absorbing~\cite{absorbing}.

\emph{(b)} Some active agent $z$ falls outside case (a): it trusts two or more agents, or a single agent that itself trusts someone. In neither case does $z$ trust exactly a core: trusting exactly a multi-agent core would make it a silent peripheral, so by Lemma~\ref{lem:reach-core} agent $z$ reaches a core member $t$; fix one, and let $T$ be its cluster. Reachability means there is a directed $M$-path $z \to z_1 \to \cdots \to t$.

\emph{Adopt up.} Walk the path forward, adopting trust: whenever $z$ trusts $z_i$ and $z_i$ trusts the next vertex $z_{i+1}$ but $z$ does not yet trust $z_{i+1}$, fire $\gos(M, z_i(z_{i+1}), z)$ --- a $0 \to 1$ move, flat by Lemma~\ref{lem:01-flat}, setting $z(z_{i+1}) := 1$. After finitely many such moves $z$ directly trusts the core member $t$.

\emph{Align.} Run the alignment sequence for $z$ with anchor $t$ (Definition~\ref{def:alignment}). A core trusts only its other members, so $t$ trusts exactly $T \setminus \{t\}$; alignment overwrites row $z$ to match $t$, giving $z(y) = 1 \iff y \in T$. Agent $z$ now trusts every member of $T$, and no longer trusts anyone outside $T$, and is therefore a peripheral of $T$.

\emph{Progress.} Because $z$ did not already trust exactly $T$, it previously reached at least one agent outside $T$ (a vertex on the path, or a surplus trustee). After zipping, $z$ trusts only $T$, and a core reaches only itself, so $z$ now reaches exactly $T$. The reachabilities from $z$ to those outside agents fall to $0$, so $|R(M)|$ strictly decreases.

Adoption and alignment edit only row $z$, so every multi-agent faction stays locked (Lemma~\ref{lem:post-processing}(a)) and no core member comes to trust $z$; thus $z$ joins $T$'s periphery without creating a new cluster. Figure~\ref{fig:zip} shows the zipping of one agent, and Figure~\ref{fig:ex-zip} completes the running example.
\end{proof}

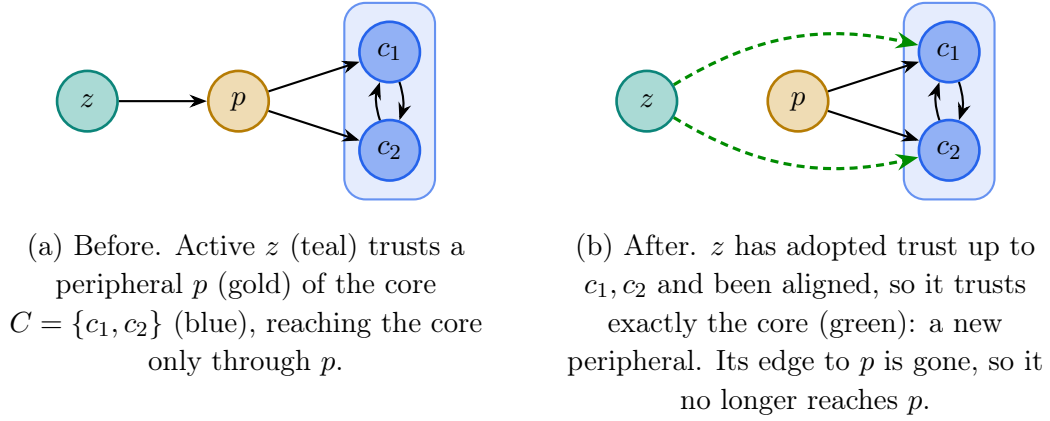
\begin{figure}[!htb]
\centering
\begin{minipage}[t]{0.46\textwidth}
\centering
\begin{tikzpicture}[
  >=Stealth, thick,
  every node/.style={circle, draw, minimum size=8mm, font=\small\bfseries, inner sep=0pt, line width=1pt}
]
\fill[factionBlue!12, rounded corners=8pt] (3.4, -1.3) rectangle (4.6, 1.3);
\draw[draw=factionBlue!70, line width=0.8pt, rounded corners=8pt] (3.4, -1.3) rectangle (4.6, 1.3);
\node[fill=factionBlue!50, draw=factionBlue] (c1) at (4, 0.65) {$c_1$};
\node[fill=factionBlue!50, draw=factionBlue] (c2) at (4, -0.65) {$c_2$};
\node[fill=factionGold!30, draw=factionGold!90!black] (p) at (2, 0) {$p$};
\node[fill=factionTeal!35, draw=factionTeal!90!black] (z) at (0, 0) {$z$};
\draw[->] (c1) to[bend left=18] (c2);
\draw[->] (c2) to[bend left=18] (c1);
\draw[->] (p) -- (c1);
\draw[->] (p) -- (c2);
\draw[->] (z) -- (p);
\end{tikzpicture}
\\[0.5em]
{\small (a) Before. Active $z$ (teal) trusts a peripheral $p$ (gold) of the core $C = \{c_1, c_2\}$ (blue), reaching the core only through $p$.}
\end{minipage}\hfill
\begin{minipage}[t]{0.46\textwidth}
\centering
\begin{tikzpicture}[
  >=Stealth, thick,
  every node/.style={circle, draw, minimum size=8mm, font=\small\bfseries, inner sep=0pt, line width=1pt}
]
\fill[factionBlue!12, rounded corners=8pt] (3.4, -1.3) rectangle (4.6, 1.3);
\draw[draw=factionBlue!70, line width=0.8pt, rounded corners=8pt] (3.4, -1.3) rectangle (4.6, 1.3);
\node[fill=factionBlue!50, draw=factionBlue] (c1) at (4, 0.65) {$c_1$};
\node[fill=factionBlue!50, draw=factionBlue] (c2) at (4, -0.65) {$c_2$};
\node[fill=factionGold!30, draw=factionGold!90!black] (p) at (2, 0) {$p$};
\node[fill=factionTeal!35, draw=factionTeal!90!black] (z) at (0, 0) {$z$};
\draw[->] (c1) to[bend left=18] (c2);
\draw[->] (c2) to[bend left=18] (c1);
\draw[->] (p) -- (c1);
\draw[->] (p) -- (c2);
\draw[->, green!55!black, very thick, densely dashed] (z) to[bend left=22] (c1);
\draw[->, green!55!black, very thick, densely dashed] (z) to[bend right=22] (c2);
\end{tikzpicture}
\\[0.5em]
{\small (b) After. $z$ has adopted trust up to $c_1, c_2$ and been aligned, so it trusts exactly the core (green): a new peripheral. Its edge to $p$ is gone, so it no longer reaches $p$.}
\end{minipage}
\caption{Zipping an active agent into a faction (Lemma~\ref{lem:zipping}). $z$ adopts the trust of agents along a path it reaches until it directly trusts a core member (flat $0 \to 1$ moves), then is aligned onto that core anchor, overwriting its row to trust exactly the core. The reachability $\Rmat_{zp}$ is destroyed, so $|R|$ drops.}
\label{fig:zip}
\end{figure}

\begin{figure}[!htb]
\centering
\begin{minipage}[t]{0.48\textwidth}
\centering
\begin{tikzpicture}[
  scale=0.72, transform shape,
  >=Stealth, thick,
  every node/.style={circle, draw, minimum size=8mm, font=\small\bfseries, inner sep=0pt, line width=1pt}
]
\fill[factionBlue!12]
  (-0.650, -0.800)
  arc (-180.00:-70.43:0.65) -- (4.718, 0.188)
  arc (-70.43:58.11:0.65) -- (2.593, 2.752)
  arc (58.11:121.89:0.65) -- (-0.343, 1.352)
  arc (121.89:180.00:0.65) -- cycle;
\draw[factionBlue!70, line width=0.8pt]
  (-0.650, -0.800)
  arc (-180.00:-70.43:0.65) -- (4.718, 0.188)
  arc (-70.43:58.11:0.65) -- (2.593, 2.752)
  arc (58.11:121.89:0.65) -- (-0.343, 1.352)
  arc (121.89:180.00:0.65) -- cycle;
\node[fill=factionBlue!50, draw=factionBlue] (n1) at (0, 0.8) {1};
\node[fill=factionBlue!50, draw=factionBlue] (n2) at (0, -0.8) {2};
\node[fill=factionGold!30, draw=factionGold!90!black] (n5) at (2.25, 2.2) {5};
\node[fill=factionMagenta!50, draw=factionMagenta] (n3) at (4.5, 0.8) {3};
\node[fill=factionMagenta!50, draw=factionMagenta] (n4) at (4.5, -0.8) {4};
\node[fill=factionTeal!40, draw=factionTeal!90!black] (n6) at (2.25, -2.2) {6};
\node[fill=gray!15, draw=gray!80!black] (n7) at (6.0, -2.2) {7};
\draw[->] (n1) to[bend left=18] (n2);
\draw[->] (n2) to[bend left=18] (n1);
\draw[->] (n5) -- (n1);
\draw[->] (n5) -- (n2);
\draw[->, green!55!black, very thick, densely dashed] (n3) -- (n1);
\draw[->, green!55!black, very thick, densely dashed] (n3) -- (n2);
\draw[->, red!75!black, very thick, densely dashed] (n3) -- (n5);
\draw[->] (n4) -- (n5);
\draw[->] (n6) -- (n3);
\draw[->] (n6) -- (n4);
\draw[->] (n6) -- (n5);
\end{tikzpicture}
\\[0.4em]
{\small (a) Zipping up agent $3$.}
\end{minipage}\hfill
\begin{minipage}[t]{0.48\textwidth}
\centering
\begin{tikzpicture}[
  scale=0.72, transform shape,
  >=Stealth, thick,
  every node/.style={circle, draw, minimum size=8mm, font=\small\bfseries, inner sep=0pt, line width=1pt}
]
\fill[factionBlue!12]
  (5.150, 0.800)
  arc (0.00:58.11:0.65) -- (2.593, 2.752)
  arc (58.11:121.89:0.65) -- (-0.343, 1.352)
  arc (121.89:180.00:0.65) -- (-0.650, -0.800)
  arc (-180.00:-121.89:0.65) -- (1.907, -2.752)
  arc (-121.89:-58.11:0.65) -- (4.843, -1.352)
  arc (-58.11:0.00:0.65) -- cycle;
\draw[factionBlue!70, line width=0.8pt]
  (5.150, 0.800)
  arc (0.00:58.11:0.65) -- (2.593, 2.752)
  arc (58.11:121.89:0.65) -- (-0.343, 1.352)
  arc (121.89:180.00:0.65) -- (-0.650, -0.800)
  arc (-180.00:-121.89:0.65) -- (1.907, -2.752)
  arc (-121.89:-58.11:0.65) -- (4.843, -1.352)
  arc (-58.11:0.00:0.65) -- cycle;
\node[fill=factionBlue!50, draw=factionBlue] (n1) at (0, 0.8) {1};
\node[fill=factionBlue!50, draw=factionBlue] (n2) at (0, -0.8) {2};
\node[fill=factionGold!30, draw=factionGold!90!black] (n5) at (2.25, 2.2) {5};
\node[fill=factionMagenta!50, draw=factionMagenta] (n3) at (4.5, 0.8) {3};
\node[fill=factionMagenta!50, draw=factionMagenta] (n4) at (4.5, -0.8) {4};
\node[fill=factionTeal!40, draw=factionTeal!90!black] (n6) at (2.25, -2.2) {6};
\node[fill=gray!15, draw=gray!80!black] (n7) at (6.0, -2.2) {7};
\draw[->] (n1) to[bend left=18] (n2);
\draw[->] (n2) to[bend left=18] (n1);
\draw[->] (n5) -- (n1);
\draw[->] (n5) -- (n2);
\draw[->] (n3) -- (n1);
\draw[->] (n3) -- (n2);
\draw[->] (n4) -- (n1);
\draw[->] (n4) -- (n2);
\draw[->] (n6) -- (n1);
\draw[->] (n6) -- (n2);
\end{tikzpicture}
\\[0.4em]
{\small (b) The absorbing state.}
\end{minipage}
\caption{Zipping process. (a) After saturate--align, $5$ is a peripheral of $\{1, 2\}$ and agent $3$ is active, trusting $5$. Adopt up makes $3$ trust the core directly (green dashed, $0 \to 1$, flat); aligning $3$ onto anchor $1$ then removes $3(5)$ (red dashed), so $3$ trusts exactly $\{1, 2\}$ and joins the periphery, dropping $|R|$. (b) Zipping $4$ and $6$ in the same way, every agent trusts exactly the core $\{1, 2\}$ except $7$, which trusts no one and is a singleton core; no firing move remains, so the state is absorbing.}
\label{fig:ex-zip}
\end{figure}
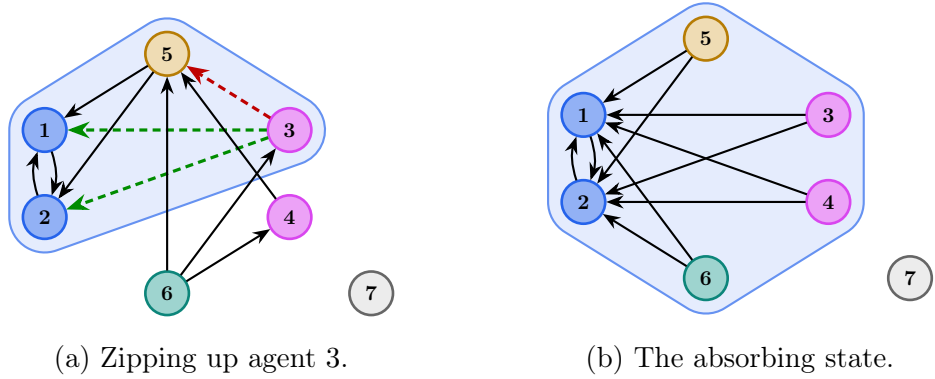


\section{Descent and termination}\label{sec:descent}

\begin{theorem}[Descent claim]\label{thm:descent}
From every non-absorbing state $M$, there exists a finite sequence of firing moves reaching either an absorbing state, or a state $M'$ with $\Rsize{M'} < \Rsize{M}$.
\end{theorem}

\begin{proof}
Either some available move at $M$ is progress, which decreases $\Rsize{M}$, or every available move is flat, so $\Rmat$ is constant on the flat-closure (Lemma~\ref{lem:R}) and the cluster structure is fixed (Corollary~\ref{cor:clusters}).

If some cluster is outward, manufacture a common trustee by flat moves (Lemma~\ref{lem:trustee}), then dissolve the cluster (Lemma~\ref{lem:B}), strictly lowering $|R(M)|$.

Otherwise every cluster is inward. Run the saturate--align process (Lemma~\ref{lem:post-processing}); its moves are flat except possibly for a $1 \to 0$ alignment step that lowers $|R(M)|$, which locks the multi-agent factions. If some active agent can still be zipped, do so (Lemma~\ref{lem:zipping}(b)): flat adoptions then an alignment which lowers $|R(M)|$. Otherwise $M$ is absorbing (Lemma~\ref{lem:zipping}(a)).
\end{proof}

\begin{corollary}[Termination]\label{thm:final-termination}
From every non-absorbing state $M$, there exists a finite sequence of firing moves leading to an absorbing state.
\end{corollary}

\begin{proof}
Apply Theorem~\ref{thm:descent} repeatedly. Each application either reaches an absorbing state or strictly drops $|R(M)|$. Since $|R(M)|$ is a non-negative integer bounded above by $n(n-1)$, and $|R(M)| = 0$ forces $M$ to be the empty trust matrix (itself absorbing), the procedure halts at an absorbing state after at most $n(n-1)$ iterations.
\end{proof}


\section{No limit cycles under fairness}\label{sec:fairness}

A \emph{firing schedule} is a sequence of firing-move choices, one per time step, where each chosen move is available at the corresponding state. A firing schedule is \emph{fair} if, whenever a state $M$ occurs infinitely often, every firing move available at $M$ is fired from $M$ infinitely often. Equivalently, a move cannot remain available at a recurring state yet never be fired.

\begin{theorem}[No limit cycles]\label{thm:no-cycles}
Under any fair firing schedule, every trajectory of the binary trust-gossip dynamics reaches an absorbing state in finitely many steps.
\end{theorem}

\begin{proof}
From every non-absorbing state an absorbing state is finitely many moves away (Corollary~\ref{thm:final-termination}); write $d(M)$ for the fewest moves in which $M$ reaches one.

Suppose towards contradiction that some fair trajectory never absorbs. The state space is finite, so some state recurs at infinitely many steps; among those, consider $N$ with $d(N)$ smallest. Since the trajectory never absorbs, $N$ is not absorbing, so $d(N) \geq 1$. The first move of a shortest run from $N$ leads to a state $N'$ with $d(N') = d(N) - 1$. That move is available whenever $N$ occurs, hence at infinitely many steps, so by fairness it fires from $N$ infinitely often, each time sending the trajectory to $N'$; thus $N'$ recurs infinitely often. If $d(N) = 1$, then $N'$ is absorbing, so the trajectory reaches an absorbing state, contradicting the assumption that it never absorbs. If $d(N) > 1$, then $N'$ is a non-absorbing recurring state with $d(N') < d(N)$, contradicting the minimality of $d(N)$.
\end{proof}


\section{Discussion}\label{sec:discussion}

Theorem~\ref{thm:no-cycles} extends the companion paper's empirical observation (no limit cycles for $n \leq 4$) to all $n$. The proof is constructive, with an explicit per-iteration cost: each descent step takes $O(n^2)$ moves: in the outward case a backward propagation followed by a dissolution, in the all-inward case the saturate--align process followed by zipping an active agent (an adoption walk plus an alignment, each $O(n)$ moves), and the outer loop runs at most $n(n-1)$ times, giving an overall bound of $O(n^4)$ moves to absorbing.

Combined with the Characterization Theorem of \cite{absorbing}, every initial state reaches one of the plane-partition-counted absorbing states.

Because a gossip move overwrites rather than adds, the dynamics are not abelian: unlike the abelian chip-firing and sandpile systems, where each starting state has one destination, termination here does not come with a unique endpoint. We prove that every trajectory halts, but which absorbing state it reaches depends on the order of gossip moves.

Three questions remain open.
\begin{enumerate}
\item \emph{A tighter bound.} From any non-absorbing state, a short run of flat moves reaches a state that is absorbing or has a progress move available. By exhaustive enumeration for $n \leq 5$ and sampling for $n = 6$, the longest such run needed is $1, 2, 4, 6$ moves at $n = 3, 4, 5, 6$. So the $O(n^2)$-per-phase factor in the $O(n^4)$ bound is likely too high and the true bound may be $O(n^3)$ or better; a proof would track which edges turn from flat to progress in each phase.
\item \emph{Basins of attraction.} Which initial states reach which absorbing state, and with what probability under a given schedule?
\item \emph{Other dynamics.} Does the argument carry over to higher-order opinions and to the other gossip types of \cite{absorbing}?
\end{enumerate}


\section*{Acknowledgements}

Formatted and illustrated using Overleaf.

No competing interests to declare.

This research did not receive any specific grant from funding agencies
in the public, commercial, or not-for-profit sectors.

\medskip
Use of AI tools: Anthropic Claude Fable 5, Opus 4.8 and Sonnet 5 were used
for assistance with drafting sections, conducting literature searches,
and drawing figures. All mathematical content was verified independently
by the author.


\begin{thebibliography}{99}

\bibitem{absorbing}
N. Boichuk, 
\emph{Absorbing states of binary trust gossip are counted by plane partitions}.
Preprint, 2026. \texttt{arXiv:2605.26792}.

\bibitem{bls} A.~Bj\"orner, L.~Lov\'asz, and P.~W. Shor,
\emph{Chip-firing games on graphs},
European J. Combin. \textbf{12} (1991), no.~4, 283--291.

\bibitem{bl-digraph} A.~Bj\"orner and L.~Lov\'asz,
\emph{Chip-firing games on directed graphs},
J. Algebraic Combin. \textbf{1} (1992), no.~4, 305--328.

\bibitem{dhar} D.~Dhar,
\emph{Self-organized critical state of sandpile automaton models},
Phys. Rev. Lett. \textbf{64} (1990), no.~14, 1613--1616.

\bibitem{bond-levine} B.~Bond and L.~Levine,
\emph{Abelian networks I: foundations and examples},
SIAM J. Discrete Math. \textbf{30} (2016), no.~2, 856--874.

\bibitem{degroot} M.~H. DeGroot,
\emph{Reaching a consensus},
J. Amer. Statist. Assoc. \textbf{69} (1974), no.~345, 118--121.

\bibitem{hk} R.~Hegselmann and U.~Krause,
\emph{Opinion dynamics and bounded confidence: models, analysis and simulation},
J. Artificial Societies and Social Simulation \textbf{5} (2002), no.~3.

\bibitem{clifford-sudbury} P.~Clifford and A.~Sudbury,
\emph{A model for spatial conflict},
Biometrika \textbf{60} (1973), no.~3, 581--588.

\bibitem{holley-liggett} R.~A. Holley and T.~M. Liggett,
\emph{Ergodic theorems for weakly interacting infinite systems and the voter model},
Ann. Probab. \textbf{3} (1975), no.~4, 643--663.

\bibitem{boyd} S.~Boyd, A.~Ghosh, B.~Prabhakar, and D.~Shah,
\emph{Randomized gossip algorithms},
IEEE Trans. Inform. Theory \textbf{52} (2006), no.~6, 2508--2530.

\bibitem{jia} P.~Jia, A.~MirTabatabaei, N.~E. Friedkin, and F.~Bullo,
\emph{Opinion dynamics and the evolution of social power in influence networks},
SIAM Rev. \textbf{57} (2015), no.~3, 367--397.


\end{thebibliography}
\end{document}